\documentclass[12pt]{amsart}
\usepackage{comment}
\usepackage{mathrsfs}
\usepackage[margin=1in]{geometry}
\usepackage{subcaption}
\usepackage{hyperref}
\hypersetup{
  colorlinks=true,
  linkcolor=blue,
  filecolor=blue,      
  urlcolor=blue,
  citecolor=blue,
  pdftitle={Overleaf Example},
}

\usepackage{amsmath}
\usepackage{amsfonts}
\usepackage{amssymb}
\usepackage{graphicx}
\usepackage{tikz-cd}
\usepackage{amsthm}
\usepackage[all]{xy}
\usepackage{adjustbox}
\usepackage{mathtools}

\newtheorem{counter}{counter}[section]
\newtheorem{theorem}[counter]{Theorem}

\newtheorem{corollary}[counter]{Corollary}
\newtheorem{lemma}[counter]{Lemma}
\newtheorem{proposition}[counter]{Proposition}

\theoremstyle{definition}
\newtheorem{definition}[counter]{Definition}
\newtheorem{conjecture}[counter]{Conjecture}
\newtheorem{notation}[counter]{Notation}
\newtheorem{remark}[counter]{Remark}

\author{Benjamin Gould and Yeqin Liu}
\date{}
\title{Cones of effective cycles on blow ups of projective spaces along rational curves}

\begin{document}
\maketitle


\begin{abstract}
    In this paper we examine the cones of effective cycles on blow ups of projective spaces along smooth rational curves. We determine explicitly the cones of divisors and 1- and 2-dimensional cycles on blow ups of rational normal curves, and strengthen these results in cases of low dimension. 

    Central to our results is the geometry of resolutions of the secant varieties of the curves which are blown up, and our computations of their effective cycles may be of independent interest.
\end{abstract}

\section{Introduction}




The theory of algebraic cycles on algebraic varieties has drawn wide and increasing attention in recent years. Because cycles of arbitrary codimension often do not exhibit properties enjoyed by divisors, their positivity properties require some care \cite{FL17a, FL17b}. For this reason, developing example classes whose cones of effective cycles can be described explicitly has proved both difficult and geometrically significant \cite{DELV11, CC15, Voi10, DJV13}. 

In this paper, we focus on the computations of the cones of effective cycles on blow ups $\mathrm{Bl}_C\mathbb{P}^r$ of projective spaces along smooth rational curves $C$. We explicitly determine the cones of effective divisors and 1- and 2-dimensional cycles when $C$ is a rational normal curve, and strengthen these results in lower-dimensional settings.

Our results progress from prior work on the cones of effective cycles on blowups of projective space along sets of points \cite{CLO16}. The intersection theory of our examples is more complicated, however, and a careful understanding of the geometry of the curve $C \subseteq \mathbb{P}^r$ is essential to our computations. 

Our main results are as follows. Let $\mathrm{Eff}^k\mathrm{Bl}_{C}\mathbb{P}^{r}$ denote the cone of effective cycles of codimension $k$ in the blow up $\mathrm{Bl}_{C}\mathbb{P}^{r}$, and similarly let $\mathrm{Eff}_k\mathrm{Bl}_{C}\mathbb{P}^{r}$ denote the cone of effective cycles of dimension $k$. Let $C$ be a rational normal curve; in this case we have $E \simeq \mathbb{P}^1 \times \mathbb{P}^{r-2}$ (see Section \ref{RNC-setup}), and we let $h_1$ and $h_2$ denote the pullbacks to $E$ of the hyperplane classes from the summands, respectively.

\begin{theorem}
Set $X_r = \mathrm{Bl}_C \mathbb{P}^r$. Let $j: E \hookrightarrow X_r$ denote the inclusion of the exceptional divisor and let $H$ denote the pullback to $X_r$ of the hyperplane class from $\mathbb{P}^r$.
    \begin{enumerate}
  \item (Proposition \ref{evendivisor} and \ref{odddivisor}) 
    $$\mathrm{Eff}^{1} (X_r)=\langle E,~ (\lfloor r/2 \rfloor +1)H- \lfloor r/2 \rfloor E \rangle, $$
    where the second class is the proper transform of the secant variety of $r/2$-planes (resp. the cone of the secant variety of $\lfloor r/2 \rfloor$-planes over a cone point on $C$) when $r$ is even (resp. odd). 
  \item (Proposition \ref{dim1})
    $$\mathrm{Eff}_{1} (X_r)=\langle j_{*} (h_{1} h_{2}^{r-3}), H^{2n}-2 j_{*} (h_{1} h_{2}^{r-3}) \rangle, $$
    where the second class is the proper transform of a $2$-secant line of $C$. 
  \item (Theorem \ref{eff2})
    $$
    \mathrm{Eff}_{2}(X_r)=
    \left\langle
    \begin{array}{c}
      (r-1)H^{r-2}-j_{*}(h_{2}^{r-3})-2rj_{*}(h_{2}^{r-4}h_{1}), \\
      H^{r-2}-3j_{*}(h_{2}^{4}h_{1}), j_{*}(h_{2}^{r-3}), j_{*}(h_{2}^{r-4}h_{1})
    \end{array}
    \right\rangle,
    $$
    where the first class is the proper transform of the cone of $C$ over a point on $C$, the second class is the proper transform of a $3$-secant plane of $C$. 
  \item (Corollary \ref{eff3}) The proper transform of the $2$-secant variety 
    $$[\mbox{Sec}_{2}(C)] = \binom{n-1}{2}H^{n-3}-(n-2)j_{*}(h_{2}^{n-4})-(n+2)(n-2)j_{*}(h_{2}^{n-5}h_{1})$$
    spans an extremal ray of $\mathrm{Eff}_{3}(X_{r})$. 
    
  \end{enumerate}
\end{theorem}

The cones of effective cycles on $X_r$ are closely related to the secant varieties $\mathrm{Sec}_m(C)$ of the blow up center $C$. When $C$ is a rational normal curve, we show that for effective cycles of dimension 2 and 3 and in codimension 1, one extremal ray is generated by the secant variety, or a projective cone over one. We conjecture that in every even (resp. odd) dimensions, the cone of effective cycles has an extremal ray generated by the secant variety (resp. the cone of the secant variety over a point on $C$) of that dimension (Conjecture \ref{extremal}). 

For smooth curves $C \subseteq \mathbb{P}^r$ which are not rational normal curves, we address many cases of geometric interest in low dimensions. We summarize these results as follows.

\begin{theorem}
    \begin{enumerate}
        \item (Proposition \ref{effW}) Let $W_r$ denote the blow up of $\mathbb{P}^r$ along a line. Then
          \[\mathrm{Eff}^k(W_r) = \langle (H-E)^k, j_*(h_2^{k-1}), j_*(h_2^{k-1}h_1) \rangle,\]
          where $(H-E)^{k}$ is the class of the proper transform of a $\mathbb{P}^{n-k}$ containing the line. 
        \item (Proposition \ref{effY} and \ref{eff1Y}) Let $Y_d$ denote the blow up of $\mathbb{P}^3$ along a curve of class $(1,d-1)$ in a smooth quadric surface. For $d\geq 3$, we have
          \[\mathrm{Eff}^1(Y_d) = \langle E, 2H-E \rangle, \quad \mathrm{Eff}_1(Y_d) = \langle j_*(h_1), H^2-(d-1)j_*(h_1)\rangle, \]
          where $2H-E$ is the class of the proper transform of the smooth quadric surface, and $H^2-(d-1)j_*(h_1)$ is the class of the proper transform of a $(d-1)$-secant line of the curve. 
        \item (Proposition \ref{effZ}) Let $S_d \subseteq \mathbb{P}^3$ be a smooth surface of degree $d$ and let $Z(C_e)$ be the blowup of $\mathbb{P}^3$ along a curve $C_e \subseteq S_d$ of degree $e$. When $e \geq d^2$,
          \[\mathrm{Eff}^1(Z(C_e)) = \langle dH-E, E \rangle, \]
          where $dH-E$ is the class of $S_{d}$. 
    \end{enumerate}
\end{theorem}

\section{Preliminaries}

In this section we collect some necessary facts and set some notation. Let $X$ be any smooth projective variety. The numerical ring $\mathrm{Num}^{*}(X)$ is defined to be the Chow ring modulo numerical equivalence. Inside $\mathrm{Num}^{k}(X)_{\mathbb{R}}=\mathrm{Num}^{k}(X)\otimes \mathbb{R}$, we define the \emph{cone of codimension-$k$ effective cycles} to be
$$\mbox{Eff}^{k}(X)=\sum \mathbb{R}_{+}\gamma,~ \gamma=[Y] \mbox{ for some codimension $k$ subvariety }Y\subset X. $$
If $\mbox{Eff}^{k}(X)$ is finite polyhedral with extremal rays generated by classes $\gamma_{1}, \cdots, \gamma_{n}$, then we denote $\mbox{Eff}^{k}(X)=\langle \gamma_{1}, \cdots, \gamma_{n} \rangle$.

Let $E$ be a vector bundle over $X$. In this paper, $\mathbb{P}(E)=\mathrm{Proj}_{X}(\mathrm{Sym}^{\bullet}(E^{\vee}))$ parametrizes 1-dimensional \emph{subspaces} of $E$. Let $f: \mathbb{P}(E) \rightarrow X$ be the projection. The tautological quotient $f^{*}(E^{\vee}) \longrightarrow \mathcal{O}_{\mathbb{P}(E)}(1)$ satisfies
$$E^{\vee} \longrightarrow f_{*}(\mathcal{O}_{\mathbb{P}(E)}(1))$$
is an isomorphism.

\subsection{Chow ring of blow ups}
Here we recall the Chow ring of the blow up of a smooth variety along a smooth center.
\begin{proposition}[\cite{EH16}, Prop. 13.12]\label{chowblowup}
  Let $i: Y \hookrightarrow X$ be an inclusion of smooth varieites and $\pi: \widetilde{X}=\mathrm{Bl}_{Y}(X) \longrightarrow X$ be the blow up. Let $j: E\cong \mathbb{P}(N_{Y/X})\hookrightarrow \widetilde{X}$ be the inclusion of the exceptional divisor, and $\xi \in \mathrm{CH}^{1}(E)$ be the class of the line bundle $\mathcal{O}_{\pi}(1)$. The Chow ring $\mathrm{CH}^{*}(\widetilde{X})$ is generated by $\pi^{*}\mathrm{CH}^{*}(X)$ and $j_{*}\mathrm{CH}^{*}(E)$ as an abelian group. The multiplicative structure is given by
  \[j_*\gamma \cdot \pi^*\alpha = j_*(\pi|_E^*i^*\alpha \cdot \gamma),~ j_*\gamma \cdot j_*\delta = -j_*(\gamma \cdot \delta \cdot \xi), \quad \alpha\in \mathrm{CH}^{*}(X), ~\gamma, \delta\in \mathrm{CH}^{*}(E).\]
\end{proposition}

\subsection{Secant varieties}

We recall some facts about the secant bundles and secant varieties of curves. Some good references are \cite{GL85, EL12, Kem21}. Let $(C, \mathcal{O}_{C}(1))$ be a smooth projective curve. The secant $k$-bundle $ \mathbb{P}(\mathcal{E}_{n,k})$ is a projective bundle over $C^{[k]}$, the Hilbert scheme of length $k$ subschemes in $C$. Over each length $k$ subscheme $ [p_{1}+\cdots +p_{k}]\in C^{(k)}$, the fiber of $\mathbb{P}(\mathcal{E}_{n, k})$ is naturally identified with the linear subspace $\overline{p_{1}\cdots p_{k}}\subset \mathbb{P}^{n}$ spanned by $p_{1}, \cdots, p_{k}$. The vector bundle $\mathcal{E}_{n,k}$ is defined as follows.

\begin{definition}\label{secantbundle}
Let $ D:=\{[p_{1}+ \cdots +p_{k}]\times r \in C^{[k]}\times C: r\in (p_{1}+\cdots + p_{k})\}$ and $ p: C^{[k]} \times C \rightarrow C^{[k]}$, $v: C^{[k]}\times C \rightarrow C$ be the projections. The secant bundle is defined to be 
$$\mathcal{E}=\mathcal{E}_{n, k}:=[p_{*}(v^{*}\mathcal{O}_{C}(n)|_{D})]^{*}. $$
\end{definition}

\begin{remark}\label{D}
  Let $D$ be as in Definition \ref{secantbundle}. The restriction map $v^{*}\mathcal{O}_{C}(n) \rightarrow v^{*}\mathcal{O}_{C}(n)|_{D}$ induces an inclusion
$$\mathbb{P}(\mathcal{E}_{n,k})=\mathbb{P}([\pi_{*}(v^{*}\mathcal{O}_{C}(n)|_{D})]^{*}) \overset{g}{\hookrightarrow} \mathbb{P}([v^{*}\mathcal{O}_{C}(n)]^{*})=C^{[k]}\times \mathbb{P}^{n}. $$
  Let $i: C \hookrightarrow \mathbb{P}^{n}$ be the inclusion. Then the following diagram is Cartesian:
  \[
\begin{tikzcd}
  D \arrow[r]\arrow[d] & C^{[k]}\times C \arrow[d, "\mathrm{id}\times i"] \\
  \mathbb{P}(\mathcal{E}_{n,k}) \arrow[r, "g"] & C^{[k]}\times \mathbb{P}^{n}.
\end{tikzcd}
  \]
\end{remark}

In this paper, when using Definition \ref{secantbundle} we always assume $C$ is a rational curve and $k<n$.

\begin{notation}\label{notationPE}
  In Definition \ref{secantbundle}, let $h$ be the hyperplane class of $ C^{[k]} \cong \mathbb{P}^{k}$, and 
  $\zeta$ be the relative $ \mathcal{O}(1)$ of $ \mathbb{P}(\mathcal{E}_{n,k})$ over $ C^{[k]}$. Let $q_{2}: C^{[k]}\times \mathbb{P}^{n} \rightarrow \mathbb{P}^{n}$ be the second projection and let $\phi=q_{2}\circ g$.
\end{notation}
The Chow rings of all $\mathbb{P}(\mathcal{E}_{n,k})$ are computable. In this paper we only need this when $C$ is a rational normal curve and $k=2$.

\begin{proposition}\label{chowsec}
   Let $C\subset \mathbb{P}^{n}$ be a rational normal curve. 
   The Chow ring of $\mathbb{P}(\mathcal{E}_{n,2})$ is
   $$ \mathrm{CH}^{*}(\mathbb{P}(\mathcal{E}_{n,2}))\cong \mathbb{Z} [h,\zeta]/(h^{3}, \zeta^{2}-(n-1)h\zeta+\binom{n}{2}h^{2}). $$
 \end{proposition}

 \begin{proof}
   Consider the following exact sequence on $C^{[k]}\times C\cong \mathbb{P}^{k}\times \mathbb{P}^{1}$: 
\begin{equation}\label{secresolve}
  0 \rightarrow I_{D}\otimes v^{*}\mathcal{O}_{C}(n) \rightarrow v^{*}\mathcal{O}_{C}(n)\rightarrow v^{*}\mathcal{O}_{C}(n)|_{D} \rightarrow 0.
\end{equation}
Note that $I_{D}\cong \mathcal{O}_{\mathbb{P}^{k}\times \mathbb{P}^{1}}(-1,-k)$ and $\mathrm{H}^{1}(C, \mathcal{O}_{C}(n-k))=0$. Applying $p_{*}$ to (\ref{secresolve}) we have
$$ 0 \longrightarrow \mathcal{O}_{\mathbb{P}^{n}}(-h)\otimes  \mathrm{H}^{0}(C, \mathcal{O}_{C}(n-k)) \longrightarrow \mathcal{O}_{\mathbb{P}^{n}}\otimes  \mathrm{H}^{0}(C, \mathcal{O}_{C}(n)) \longrightarrow \mathcal{E}_{n,k}^{*} \longrightarrow 0. $$
Hence the total Chern class of $\mathcal{E}_{n,k}$ is
$$c(\mathcal{E}_{n,k})=\dfrac{1}{(1+h)^{n-k+1}}=\left[\sum_{j=0}^{\infty}(-h)^{j}\right]^{n-k+1}.$$
The proposition follows from the formula for the Chow ring of a projective bundle (e.g. \cite{EH16}, Theorem 9.6).
 \end{proof}

 We also need to know the cone of effective divisors of $ \mathbb{P}(\mathcal{E}_{n,2})$.

\begin{proposition}\label{eff1PE}
 Let $C\subset \mathbb{P}^{n}$ be a rational normal curve. Using Notation \ref{notationPE}, we have
  $$\mathrm{Eff}^{1}(\mathbb{P}(\mathcal{E}_{n,2}))=\langle h, D \rangle
  =\langle h, 2\zeta - (n-2)h\rangle. $$ 
\end{proposition}

\begin{proof}
Note that $h$ is contracted by the projection $\mathbb{P}(\mathcal{E}_{n,2}) \rightarrow C^{[k]}$, and $D$ is contracted by $\phi$ since $\phi(D)=C\subset \mathbb{P}^{n}$. Hence $h, D$ span extremal rays. Since the restriction of $D$ to every fiber of $\pi$ consists of 2 points, we may write $D=2\zeta - mh$ for some $m\in \mathbb{Z}$. Note that $ \zeta=\phi^{*}\mathcal{O}_{\mathbb{P}^{n}}(H)$, we have
$$ D \cdot \zeta^{2}= D\cdot (\phi^{*}H^{2})=(\phi_{*}D) \cdot H^{2}=0.$$
By Proposition \ref{chowsec}, we have $m=n-2$.
\end{proof}

The secant variety $\text{Sec}_{m}(C) \subseteq \mathbb{P}^{r}$ of $C$ is swept out by $(m-1)$-planes $\mathbb{P}^{m-1} \subseteq \mathbb{P}^r$ which meet $C$ in $m$ points. This can be defined formally.

\begin{definition}\label{secantvariety}
Using Notation \ref{notationPE}, the secant variety of $(m-1)$-planes of $C\subset \mathbb{P}^{n}$ is
$$\mbox{Sec}_{m}(C)=\phi(\mathbb{P}(\mathcal{E}_{n, m})). $$
\end{definition}

\section{Blow up along Rational Normal Curves} \label{RNC-setup}
Every non-degenerate rational curve in a projective space is the image of a non-degenerate rational normal curve in a projective space of higher dimension under a linear projection. In this section, we study the cones of effective cycles where the blow up center is a non-degenerate rational normal curve. In order to present the main results we need to set up necessary notations. 

\subsection{Set-up and notation.} Let $C\subset \mathbb{P}^{r}$ be a rational normal curve, and $\pi: X_{r} = \text{Bl}_C \mathbb{P}^r \rightarrow \mathbb{P}^r$ be the blow up. The normal bundle of $C \subseteq \mathbb{P}^r$ is $N_{C/\mathbb{P}^r} \cong \mathcal{O}_{\mathbb{P}^1}(r+2)^{r-1}$ \cite{CR18}. The exceptional divisor of $j: E\hookrightarrow X_{r}$ can be described as $E=\mathbb{P} N_{C/\mathbb{P}^r} \cong \mathbb{P}^1 \times \mathbb{P}^{r-2}$. We have
$$\mathrm{Pic}(E)=\mathrm{Pic}(\mathbb{P}^{1}\times \mathbb{P}^{r-2})\cong \mathbb{Z}h_{1}\oplus \mathbb{Z}h_{2},$$
where $h_1, h_2 \in \text{Pic}(E)$ are the pullbacks of the hyperplane classes from $\mathbb{P}^1$ and $\mathbb{P}^{r-2}$ respectively. By Proposition \ref{chowblowup}, the Chow ring of $X_{r}$ is computed in the following proposition. 

\begin{proposition}\label{chowRNC}
  We have $-\xi=c_{1}(N_{E/X_{r}})=-h_{2}+(r+2)h_{1}\in \mathrm{Pic}(E)$. The Chow ring of $X_{r}$ is
  \[
    \mathrm{CH}^{*}(X_{r})=
    \dfrac{
      \left[ \dfrac{\mathbb{Z}[H]}{(H^{r+1})} \right]
      \oplus
      j_{*}\left[\dfrac{\mathbb{Z}[h_{1}, h_{2}]}{(h_{1}^{2}, h_{2}^{r-1})}\right]
    }
    {
      \left(
        \begin{array}{c}
          H\cdot j_{*}(\alpha)-j_{*}(\alpha \cdot (r+1)h_{1}), \\
          j_{*}(\beta)\cdot j_{*}(\gamma)+j_{*}(\xi \cdot \beta \cdot \gamma)
        \end{array}
      \right)
    },
    \quad
    \alpha, \beta, \gamma \in \dfrac{\mathbb{Z}[h_{1}, h_{2}]}{(h_{1}^{2}, h_{2}^{r-1})}. 
  \]
\end{proposition}
Note that $\mathrm{CH}^{1}(X_{r})$ has rank 2 and $\mathrm{CH}_{1}(X_{r})$ has rank 3, there is a numerical relation on $\mathrm{NS}_{1}(X_{r})$. By Proposition \ref{chowRNC}, we have the following corollary that will be used later.
\begin{corollary}\label{relation}
  In $\mathrm{Num}_{1}(X_{r})$, we have
$$ h_{2}^{r-2}=rH^{r-1}-(r+2)j_{*}(h_{2}^{r-3}h_{1}). $$
\end{corollary}

Now we present our main result.

\begin{theorem}\label{effRNC}
  For every $r\geq 3$, using the presentation of $\mathrm{CH}^{*}(X_{r})$ in Proposition \ref{chowRNC}, we have
  \begin{enumerate}
  \item (Proposition \ref{evendivisor} and \ref{odddivisor})
    $$\mathrm{Eff}^{1} (\mathrm{Bl}_{C}\mathbb{P}^{r})=\langle E,~ (\lfloor r/2 \rfloor +1)H- \lfloor r/2 \rfloor E \rangle, $$
  \item (Proposition \ref{dim1})
    $$\mathrm{Eff}_{1} (\mathrm{Bl}_{C}\mathbb{P}^{r})=\langle j_{*} (h_{1} h_{2}^{r-3}), H^{2n}-2 j_{*} (h_{1} h_{2}^{r-3}) \rangle, $$
  \item (Theorem \ref{eff2})
    $$\mathrm{Eff}_{2} X_{r}=\langle (r-1)H^{r-2}-j_{*}(h_{2}^{r-3})-2rj_{*}(h_{2}^{r-4}h_{1}), H^{r-2}-3j_{*}(h_{2}^{4}h_{1}), j_{*}(h_{2}^{r-3}), j_{*}(h_{2}^{r-4}h_{1}) \rangle .$$
  \end{enumerate}
\end{theorem}

\subsection{Effective divisors}

In this subsection we compute the cone of effective divisors of $X_{r}$, which is Theorem \ref{effRNC} (1). Since $\mbox{Pic}(X_{r})=\mathrm{Num}^{1}(X_{r})= \mathbb{Z}H\oplus \mathbb{Z}E$, it suffices to find the two extremal rays. 

For every $m\geq 1$, let $S_{m}=S_{m}(C)$ be the proper transform of $\mathrm{Sec}_{m}(C)$ (Definition \ref{secantvariety}). When $r = 2n$ is even, we have $\dim(\mathrm{Sec}_n(C)) = 2n-1$. In this case, the cone of effective divisors is generated by $E$ and $S_{n}(C)$. 

\begin{proposition}\label{evendivisor}
For every $n\geq 1$, $\mathrm{Eff}^{1}(X_{2n})=\langle E, [S_{n}(C)] \rangle $.
\end{proposition}

\begin{proof}
  Since $E$ is contracted by $\pi$, it spans an extremal ray.
To show $[\mathrm{Sec}_{2}(C)]$ spans an extremal ray, it suffices to show its dual class $nH^{2n}-(n+1)j_{*} (h_{1} h_{2}^{2n-2})$ is nef. This is done if we show the curves of degree $n$ that pass through $(n+1)$ points on $C$ sweep out the whole projective space $\mathbb{P}^{2n}$.
  When $n=1$, we take the degree $1$ curve to be any secant line of $C$. Since $\mathrm{Sec}_{2}(C)=\mathbb{P}^{2}$, $H^{2}-2j_{*}(h_{1})$ is nef. 
  When $n\geq 2$, let $\Sigma$ be a union of $\lceil (n+1)/2 \rceil$ lines containing the $(n+1)$ points. Then $d:=n-\lceil (n+1)/2 \rceil>0$. We may take the degree $n$ curve to be $\Sigma\cup Z$ for any curve $Z$ of degree $d$, then $\mathbb{P}^{2n}$ is swept out by $Z$.
\end{proof}

In order to write $\mbox{Eff}^{1}(X_{2n})$ in terms of the presentation of Proposition \ref{chowRNC}, we need to compute the class of $S_{n}(C)$ in $\mbox{Pic}(X_{2n})$.

\begin{proposition}\label{evensec}
Let $C\subset \mathbb{P}^{2n}$ be a rational normal curve. Then $[S_n(C)] = (n+1)H - nE $. 
\end{proposition}

\begin{proof}
By \cite{EH16} Theorem 10.16, $\deg(\text{Sec}_{n}(C)) = n+1$ for all $n$. The coefficient of $E$ is the generic multiplicity of $Sec_{n}(C)$ along $C$. We compute this by induction on $n$.

First consider $n=2$. Pick a general line $L_{1}$ that meets $C$ in a general point $p$, and let $\pi_{1}:\mathbb{P}^{4}\dashrightarrow \mathbb{P}^{2}$ be the projection from $L_{1}$. Then $\pi_{1}(C)$ is a rational curve of degree 3, which has 1 node. The node corresponds to $(\mathrm{Sec}_{2}(C)\cap L_{1})\setminus C$. Since $\mathrm{Sec}_{2}(C)\cdot L_{1}$ has length 3, the intersection multiplicity $\mathrm{Sec}_{2}(C)\cdot L_{1}$ at $p$ is 2. Hence $[\mathrm{Sec}_{2}(C)]=3H-2E$. 

Suppose the lemma is proved for $X_{2(n-1)}$, we now prove it for $X_{2n}$. Pick a general line $L_{2}$ that meets $C$ in a general point and let $\pi_{2}:\mathbb{P}^{2n} \dashrightarrow \mathbb{P}^{2n-2}$ be the projection from $L_{2}$. The points in $(L_{2}\cap \mathrm{Sec}_{n}(C))\setminus C$ correspond to $n$-secant $\mathbb{P}^{n-2}$ of $\pi_{3}(C)\subset \mathbb{P}^{2n-2}$, denote the number of such points by $\delta$. Then $[S_{n}(C)]=(n+1)H-(n+1-\delta)E$, $\delta\geq 1$. Note that $\pi_{3}(C)$ is a rational curve of degree $2n$, it specializes to $C'\cup L$, where $C'$ is a rational normal curve in $\mathbb{P}^{2n-2}$ and $L$ is a general line intersecting $C'$ at a point. Denote the number of $n$-secant $\mathbb{P}^{n-2}$ of $C'\cup L$ by $\delta'$, then $\delta\leq \delta'$. Note that $\delta'=\#[(\mathrm{Sec}_{n-1}(C')\cap L)\setminus C']$, by induction hypothesis $\delta'=1$. Hence $\delta=1$. 

\end{proof}

In the following proposition we compute $\mbox{Eff}^{1}(X_{r})$ when $r$ is odd. The technique is similar to that of Proposition \ref{evendivisor}. 

\begin{proposition}\label{odddivisor}
  For every $n\geq 1$, $\mathrm{Eff}^{1}(X_{2n+1})=\langle E, (n+1)H-nE \rangle.$
\end{proposition}

\begin{proof}
Let $x \in C \subset \mathbb{P}^{2n+1}$ be a point, $\pi_x: C \dashrightarrow P$ be the projection from $x$ to a complementary hyperplane $P$, and $Y$ be the cone of $\text{Sec}_{n-1}(\pi_x(C))\subset P$ with vertex $x$. Note that $\pi_x(C)$ is a rational normal curve in $P \cong \mathbb{P}^{2n}$ and $Y$ is a cone over $\mathrm{Sec}_{n-1}(\pi_{x}(C))\subset \mathbb{P}^{2n}$, by Lemma \ref{evensec} the degree of $Y$ is $(n+1)$ and the generic multiplicity of $Y$ along $C$ is $n$. Hence $[Y]=(n+1)H-nE$. Since $E$ is contracted by $\pi$, it spans an extremal ray. By a similar argument as in the proof of Proposition \ref{evendivisor}, we see that $[Y]$ spans another extremal ray.
\end{proof}

\subsection{One dimensional effective cycles}

The cone of one dimensional effective cycles of $\mathrm{BL}_{C}\mathbb{P}^{r}$ is given by the following proposition, which is Theorem \ref{effRNC} (2). 

\begin{proposition}\label{dim1}
  For every $r\geq 2$, $\mathrm{Eff}_{1} (\mathrm{Bl}_{C}\mathbb{P}^{r})=\langle j_{*} (h_{1} h_{2}^{r-3}), H^{2n}-2 j_{*} (h_{1} h_{2}^{r-3}) \rangle$.
\end{proposition}

\begin{proof}
The class $j_{*} (h_{1} h_{2}^{r-3})$ is contracted by $\pi$, hence it spans an extremal ray. The dual class of $H^{2n}-2 j_{*} (h_{1} h_{2}^{r-3})$ is $2H-E$. Since the ideal sheaf of $C$ is generated by quadrics, $2H-E$ is nef. Hence $H^{2n}-2j_{*}(h_{1}h_{2}^{r-3})$ spans another extremal ray. 
\end{proof}

\subsection{Dimension Two Effective Cycles}
In this section we compute the cones of two-dimensional cycles on blowups of rational normal curves in $\mathbb{P}^r$, which is Theorem \ref{effRNC} (3). Let $ \mathcal{E}=\mathcal{E}_{n,2}$ be the secant bundle of $C\subset \mathbb{P}^{n}$. We have the following lemma.

\begin{lemma}\label{S2}
We have $S_{2}\cong \mathbb{P}(\mathcal{E})$.
\end{lemma}

\begin{proof}
Each of the two birational morphisms $ \pi: S_{2} \rightarrow \mathbb{P}^{n}$ and 
$ \phi: \mathbb{P}\mathcal{E} \rightarrow \mathbb{P}^{n}$ contracts all fibers that the other
contracts.
\end{proof}

In order to write $\psi_{*}\mathrm{CH}^{*}(S_{2})$ in terms of generators of $\mathrm{CH}^{*}(X_{n})$ in Proposition \ref{chowRNC}, we need the following lemma.

\begin{lemma}\label{psi*}
  Using Notation \ref{notationPE}, we have
  \begin{eqnarray*}
    \psi^{*}H& = &\zeta, \\
\psi^{*}E& = &D=2\zeta -(n-2)h, \\
\psi^{*}j_{*}(h_{1})& = &\zeta h-(n-1)h^{2}, \\
    \psi^{*}j_{*}(h_{2})& = &(n-2)(\zeta h +h^{2}).
  \end{eqnarray*}
\end{lemma}

\begin{proof}
  The first equality is because
  $$\zeta=\phi^{*}\mathcal{O}_{\mathbb{P}^{n}}(H)=\psi^{*}\pi^{*}\mathcal{O}_{\mathbb{P}^{n}}(H)= \psi^{*}H. $$
  The second equality is by definition of $D$ and Proposition \ref{eff1PE}. 
  The third equality can be shown by noticing
  $$ \psi^{*}(nj_{*}(h_{1}))=\psi^{*}(H\cdot E)=\zeta\cdot(2\zeta -(n-2)h)=n\zeta h-(n-1)n h^{2}. $$ 
  For the last equality, note that we may compute $\psi^{*}(E^{2})$ in two ways by using the second equality and Proposition \ref{chowRNC}:
    \begin{eqnarray*}
      \psi^{*}(E^{2}) & = &(2\zeta -(n-2)h)^2=4\zeta h -(n^{2}+2n-4)h^{2},\\
      \psi^{*}(E^{2}) & = &\psi^{*}((n+2)j_{*}(h_{1})-j_{*}(h_{2}))=(n+2)\zeta h-(n+2)(n-1)h^{2}-\psi^{*}j_{*}(h_{2}).
    \end{eqnarray*}
Comparing coefficients, we have $ \psi^{*}j_{*}(h_{2})=(n-2)(\zeta h + h^{2})$.
\end{proof}

\begin{lemma}\label{push}
  Using Notation \ref{notationPE}, we have
  \begin{eqnarray*}
    \psi_{*}h & = &(n-1)H^{n-2}-j_{*}(h_{2}^{n-3})-2nj_{*}(h_{2}^{n-4}h_{1}),\\
    \psi_{*}[S_{2}] & = &\binom{n-1}{2}H^{n-3}-(n-2)j_{*}(h_{2}^{n-4})-(n+2)(n-2)j_{*}(h_{2}^{n-5}h_{1}),\\
    \psi_{*}\zeta & = &\binom{n-1}{2}H^{n-2} -n(n-2)j_{*}(h_{2}^{n-4}h_{1}).
  \end{eqnarray*}
\end{lemma}

\begin{proof}
  Note that in $X_{n}$ a cone of $ C$ with cone point on $ C$ has class $ \psi_{*}h$. Hence we may write $ \psi_{*}h=(n-1)H^{n-2}-aj_{*}(h_{2}^{n-3})-bj_{*}(h_{2}^{n-4}h_{1})$ for some $a,b \in \mathbb{Z}$. By Proposition \ref{chowRNC}, Proposition \ref{chowsec} and Lemma \ref{psi*}, we have
\begin{eqnarray*}
  a&=&\psi_{*}h\cdot j_{*}(h_{1})=h\cdot \psi^{*}j_{*}(h_{1})=h\cdot(\zeta h-(n-1)h^{2})=1,\\
  b-(n+2) &=& \psi_{*}h\cdot j_{*}(h_{2})=h\cdot \psi^{*}j_{*}(h_{2})=h\cdot((n-2)\zeta h+(n-2)h^{2})=n-2.
\end{eqnarray*}
Next, we write $\psi_{*}[S_{2}]=dH^{n-3}-ej_{*}(h_{2}^{n-4})-fj_{*}(h_{2}^{n-5}h_{1})$ for some $d,e,f\in \mathbb{Z}$. By Proposition \ref{chowRNC}, Proposition \ref{chowsec} and Lemma \ref{psi*}, we have
\begin{eqnarray*}
  d &= &\psi_{*}[S_{2}]\cdot H^{3}=\psi_{*}((\psi^{*}H)^{3})=\psi_{*}(\zeta^{3})=\binom{n-1}{2},\\
  ne & = & \psi_{*}[S_{2}]\cdot j_{*}(nh_{1}h_{2})=\psi_{*}[S_{2}]\cdot (j_{*}(h_{2})\cdot H)\\
    & = &\psi_{*}(\psi^{*}j_{*}(h_{2})\cdot \psi^{*}(H)) =\psi_{*}[(n-2)(\zeta h +h^{2})\cdot \zeta]=n(n-2),\\
  2(n+2)e-f&= &\psi_{*}[S_{2}]\cdot (j_{*}(h_{2})\cdot E)=\psi_{*}(\psi^{*}j_{*}(h_{2})\cdot \psi^{*}(E))\\
  &=&\psi_{*}[(n-2)(\zeta h + h^{2})(2\zeta - (n-2)h)]=(n-2)(n+2).
\end{eqnarray*}
Hence $d=\binom{n-1}{2}, e=n-2, f=(n+2)(n-2)$.
Finally, we have
\[\psi_{*}\zeta=\psi_{*}(\psi^{*}H)=\psi_{*}[S_{2}]\cdot H=\binom{n-1}{2}H -n(n-2)j_{*}(h_{2}^{n-4}h_{1}).\]
\end{proof}

\begin{lemma}\label{nefn-1}
  The class $H^{2}-(n-1)j_{*}(h_{1})$ is nef.
\end{lemma}

\begin{proof}
We need to show that for any irreducible surface $ Y\subset X_{n}$, $(H^{2}-(n-1)h_{1})\cdot [Y]\geq 0$. If $Y\subset E$, then $ [Y]=j_{*}(ah_{2}^{n-3}+bh_{2}^{n-4}h_{1})$ with $ a,b \geq 0$. We have
$$(H^{2}-(n-1)j_{*}(h_{1}))\cdot [Y]=(H^{2}-(n-1)j_{*}(h_{1}))\cdot j_{*}(ah_{2}^{n-3}+bh_{2}^{n-4}h_{1})=(n-1)a\geq 0. $$

If $ Y \nsubseteq E$, $\pi_{*}(Y)\subset \mathbb{P}^{n}$ is an irreducible surface. We want to show there exist an $ (n-1)$-secant $L\cong \mathbb{P}^{n-2}\subset \mathbb{P}^{n}$, whose proper transform intersects $ Y$ dimensionally transversely.

If not, the intersection
 contains a curve $R=R(L)$. Note that $(H^{2}-(n-1)j_{*}(h_{1}))\cdot E = (n-1)j_{*}(h_{1}h_{2})$ is base point free, for a general $L$ we have $R(L)\not\subset E$. Hence $\pi_{*}(R(L))$ is a curve contained in $L\cap \pi_{*}(Y)$ for a general $L$. By upper semi-continuity of dimension, $L\cap \pi_{*}(Y)$ contains a curve for every $L$. The incidence correspondence
 $$ \Gamma:=\{(p, L)\in \mathbb{P}^{n}\times C^{[n-1]}: L \text{ is $(n-1)$-secant to }C, p\in L\cap \pi_{*}(Y)\}. $$
 has $\mathrm{dim} \Gamma =n$. Let $ p_{1}: \Gamma \rightarrow \mathbb{P}^{n}$ be the first projection. If $ p_{1}(\Gamma)$ has dimension 1, then for every point $ p\in p_{1}(\Gamma)$, all $(n-1)$-secant $L$ contain $p$, which is impossible. Hence $ p_{1}(\Gamma)=\pi_{*}(Y)$ and a general fiber of $p_{1}$ has dimension $(n-2)$. For $p\in \mathbb{P}^{n}$, let
\[
  \begin{array}{c}
    \sigma_{2}(p):=\{[L]\in G(n-1,n+1)|p\in L\},\\
    \Psi_{n-1}:=\{L\in G(n-1,n+1)| L \text{ is $(n-1)$-secant to }C\}.
  \end{array}
\]
Then for $p\in \pi_{*}(Y)$ we have $p_{1}^{-1}(p)=\sigma_{2}(p)\cap \Psi_{n-1}$. If we show that $ \sigma_{2}(p)$ and $ \Psi_{n-1}$ intersect transversely when $ p \notin C$, then a general fiber of $p_{1}$ would have dimension $(n-3)$, a contradiction.
Take $ L \in \sigma_{2}(p)\cap \Psi_{n-1}$ and assume $L\cap C$ consists of the first $ (n-1)$-coordinate points. There exists a $2\times n$ matrix
$$
M=
\begin{pmatrix}
  v_{1,1} & v_{1,2} & \cdots & v_{1,n-1}\\
  v_{2,1} & v_{2,2} & \cdots & v_{2,n-1}
\end{pmatrix},
$$
up to non-zero scaling of each column, whose columns represent tangent vectors of $C$ at $C\cap L$. Since $C$ is a rational normal curve, every $2\times 2$ minor of $M$ has full rank. Using $M$, we may write
$$ T_{[L]}\Psi_{n-1}=\left\{A\in \mathrm{Mat}_{2\times n}: A=
M\cdot (\lambda_{1}, \cdots, \lambda_{n-1})^{t}, 
\lambda_{i}\in \mathbb{C}\right\}.$$
Write $p=[x_{0}, \cdots, x_{n}]\in \mathbb{P}^{n}$, then
$$T_{[L]}\sigma_{2}(p)=\{A\in \mathrm{Mat}_{2\times n}:A(x_{0}, \cdots, x_{n})^{t}=0\}.$$
 Since $p\not\in C$, $ p$ is not a coordinate point in $ L$. Hence there are at least 2 non-zero coordinates of $p$, say $x_{0}, x_{1}\neq 0$. Then for $A\in T_{[L]}\Psi_{n-1}\cap T_{[L]}\sigma_{2}(p)$, we have $ M\cdot (\lambda_{1}, \lambda_{2}, 0, \cdots, 0)^{t}=0$. Since every $2\times 2$ minor of $M$ has full rank, $ \lambda_{1}=\lambda_{2}=0$. Hence $\mathrm{dim}(T_{[L]}\Psi_{n-1}\cap T_{[L]}\sigma_{2}(p))=n-3$.
\end{proof}

\begin{remark}\label{notnefn-2}
The class $H^{3}-(n-2)j_{*}(h_{1}h_{2})$ is not nef: its intersection with $[S_{2}]$ is negative. By a similar argument as the proof of Lemma \ref{nefn-1}, we can show that $[S_{2}]$ is the only irreducible threefold whose intersection with $H^{3}-(n-2)j_{*}(h_{1}h_{2})$ is nagetive. 
\end{remark}

\begin{corollary}\label{eff3}
  The class $[S_{2}]$ spans an extremal ray of $\mathrm{Eff}_{3}(X_{n})$.
\end{corollary}

\begin{proof}
Let $ H^{n-3}-ah_{2}^{n-4}-bh_{2}^{n-5}h_{1}$ be in the ray spanned by an irreducible threefold $ Y$. By Remark \ref{notnefn-2}, if $ Y\neq S_{2}$, then $ [Y]\cdot (H^{3}-(n-2)h_{1}h_{2})\geq 0$, we have $ a\leq 1/(n-2)$. However, $[S_{2}]=\binom{n-1}{2}H^{n-3}-(n-2)j_{*}(h_{2}^{n-4})-(n+2)(n-2)j_{*}(h_{2}^{n-5}h_{1})$ does not satisfy this inequality.
\end{proof}

Now we are ready to compute the cone of 2-dimensional effective cycles of $X_{n}$.

\begin{theorem}\label{eff2}
  For every $n\geq 3$, we have
$$\mathrm{Eff}_{2} X_{n}=\langle (n-1)H^{n-2}-j_{*}(h_{2}^{n-3})-2nj_{*}(h_{2}^{n-4}h_{1}), H^{n-2}-3j_{*}(h_{2}^{4}h_{1}), j_{*}(h_{2}^{n-3}), j_{*}(h_{2}^{n-4}h_{1}) \rangle. $$
\end{theorem}

\begin{proof}
Take any irreducible surface $ S\subset X_{n}$ that is not contained in $E$. The ray it generates has a unique class $\gamma(S)=H^{n-2}-aj_{*}(h_{2}^{n-3})-bj_{*}(h_{2}^{n-4}h_{1}), a,b \geq 0$.

By Lemma \ref{nefn-1}, $H^{2}-(n-1)j_{*}h_{1}$ is nef,
$$0\leq (H^{2}-(n-1)j_{*}(h_{1}))\cdot \gamma(S)=1-a(n-1).$$
Hence $a\leq 1/(n-1)$. 

Let $L\subset \mathbb{P}^{n}$ be an $(n-4)$-secant linear space $L\cong \mathbb{P}^{n-5}$ of $C$, $P\subset \mathbb{P}^{n}$ be a complementary $\mathbb{P}^{4}$ of $L$, and $\overline{C}\subset P$ be the image of the projection of $C$ from $L$. Let $S_{2}(L)$ be proper transform of the cone of $\mathrm{Sec}_{2}(C\subset \mathbb{P}^{n})$ over $L$. Then $S_{2}(L)$ has class $3H-2E$, since it is the proper transform of the cone of $\mathrm{Sec}_{2}(\overline{C}\subset P)$ over $L$. By moving $L$, we see that the base locus of $3H-2E$ is contained in $\bigcap_{L}S_{2}(L)=S_{2}$. 

If $S\not\subset S_{2}$, by Corollary \ref{relation} the class
$$\gamma(S)\cdot (3H-2E)=(3-2na)H^{n-1}-(2b-(n+8)a)j_{*}(h_{2}^{n-3}h_{1})$$
is a 1-dimensional effective cycle in $\mathrm{NS}_{1}(X_{n})$. By Proposition \ref{dim1}, we have
\begin{equation*}\label{b1}
  \dfrac{2b-(n+8)a}{3-2na}\leq 2,
\end{equation*}
which is equivalent to $2b+(3n-8)a\leq 6$.

By Proposition \ref{eff1PE} and Lemma \ref{push}, we have
\begin{eqnarray*}
  \psi_{*}\mathrm{Eff}_{2}(S_{2})&=&\langle \psi_{*}(h), \psi_{*}(2\zeta -(n-2)h)\rangle \\
  &=&\langle (n-1)H^{n-2}-j_{*}(h_{2}^{n-3})-2nj_{*}(h_{2}^{n-4}h_{1}), (n-2)j_{*}(h_{2}^{-3}) \rangle.
\end{eqnarray*}
Let $A$ be the cone
$$\langle (n-1)H^{n-2}-j_{*}(h_{2}^{n-3})-2nj_{*}(h_{2}^{n-4}h_{1}), H^{n-2}-3j_{*}(h_{2}^{4}h_{1}), j_{*}(h_{2}^{n-3}), j_{*}(h_{2}^{n-4}h_{1}) \rangle, $$
and $B$ be the cone spanned by 
$$\left\{H^{n-2}-ah_{2}^{n-3}-bh_{2}^{n-4}h_{1}\bigg|
  \begin{array}{l}
    0\leq a, b, a\leq 1/(n-1),\\
  2b+(3n-8)a\leq 6
  \end{array}
\right\}. $$
If $S\not\subset S_{2}$, then $\gamma(S)\in B$. If $S\subset S_{2}$, then $\gamma(S)\in A$. Note that $B\subset A$, we have $\mathrm{Eff}_{2}(X_{n})\subset A$. Since the 4 extremal rays of $A$ are all effective, we have $A\subset \mathrm{Eff}_{2}(X_{n})$. 
\end{proof}

Note that in Theorem \ref{eff2}, the class $(n-1)H^{n-2}-j_{*}(h_{2}^{n-3})-2nj_{*}(h_{2}^{n-4}h_{1})$ is the proper transform of the cone of $C$ over a point on $C$. It generates an extremal ray in $\mbox{Eff}_{2}(X_{n})$. Motivated by this fact, Proposition \ref{evendivisor}, Proposition \ref{odddivisor}, and Corollary \ref{eff3}, we have the following conjecture.

\begin{conjecture}\label{extremal}
For every $2\leq d \leq n/2$, the proper transform of $\mbox{Sec}_{d}(C)$ generates an extremal ray of $\mbox{Eff}_{2d-1}$. For every $1\leq d < n/2$, the proper transform of the cone of $\mbox{Sec}_{d}(C)$ over a point on $C$ generates an extremal ray of $\mbox{Eff}_{2d}$. 
\end{conjecture}

\section{Blow up along other rational curves}

In this section we study the cones of effective cycles in $\mathrm{BL}_{C}(\mathbb{P}^{r})$ for some other curves $C\subset \mathbb{P}^{r}$.

\subsection{Blow up projective spaces along lines}

Let $\pi: W_{r}=\mathrm{BL}_{L}\mathbb{P}^{r} \rightarrow \mathbb{P}^{r}$ where $L\subset \mathbb{P}^{r}$ is a line, and let $v: W_{r} \rightarrow \mathbb{P}^{r-2}$ be the resolution of the projection from $L$ to a complementary $\mathbb{P}^{r-2}$. We have $N_{L/\mathbb{P}^{r}}\cong \mathcal{O}_{\mathbb{P}^{1}}(1)^{\oplus r-1}$ and the exceptional divisor $j: E\hookrightarrow W_{r}$ can be described as $E= \mathbb{P}(N_{L/\mathbb{P}^{r}})\cong \mathbb{P}^{1}\times \mathbb{P}^{r-2}$. Hence $\mathrm{CH}^{*}(E)\cong \mathbb{Z}[h_{1}, h_{2}]/(h_{1}^{2}, h_{2}^{r-1})$, where $h_{1}$ is the hyperplane class of $L$ and $h_{2}$ is the hyperplane class of $\mathbb{P}^{r-2}$. By Proposition \ref{chowblowup}, the Chow ring of $W_{r}$ is computed in the following proposition.

\begin{proposition}\label{chowW}
  We have $-\xi=c_{1}(N_{L/\mathbb{P}^{r}})=-h_{2}+h_{1}\in \mathrm{Pic}(E)$. The Chow ring of $W_{r}$ is
  \[
    \mathrm{CH}^{*}(W_{r})=
    \dfrac{
      \left[ \dfrac{\mathbb{Z}[H]}{(H^{r+1})} \right]
      \oplus
      j_{*}\left[\dfrac{\mathbb{Z}[h_{1}, h_{2}]}{(h_{1}^{2}, h_{2}^{r-1})}\right]
    }
        {\left(
      \begin{array}{c}
        H\cdot j_{*}(\alpha)-j_{*}(\alpha\cdot h_{1}),\\
        j_{*}(\beta)\cdot j_{*}(\gamma)+j_{*}(\xi\cdot \beta \cdot \gamma)
      \end{array}
    \right) },
  \quad \alpha, \beta, \gamma \in \dfrac{\mathbb{Z}[h_{1}, h_{2}]}{(h_{1}^{2}, h_{2}^{2})}.
  \]
\end{proposition}

The cones of effective cycles of $W_{r}$ of all dimensions are generated by linear spaces.

\begin{proposition}\label{effW}
  For every $r\geq 2$ and $1\leq k \leq r-1$, we have
  $$\mathrm{Eff}^{k}(W_{r})=\langle (H-E)^{k}, j_{*}(h_{2}^{k-1}), j_{*}(h_{2}^{k-2}h_{1}) \rangle. $$
\end{proposition}

\begin{proof}
  When $k=1$, $E$ is contracted by $\pi$ and $H-E$ is contracted by $v$, hence they span extremal rays of $\mathrm{Eff}^{1}(W_{r})$. When $k=r-1$, $(H-E)^{r-1}=0$. Since $j_{*}(h_{2}^{r-3}h_{1})$ is contracted by $\pi$ and $j_{*}(h_{2}^{r-2})$ is contracted by $v$, they span extremal rays of $\mathrm{Eff}_{1}(W_{r})$. In the following we assume $2\leq k \leq r-2$.

  The dual class of $\langle j_{*}(h_{2}^{k-1}), j_{*}(h_{2}^{k-2}h_{1}) \rangle$ is $H^{r-k}$. The dual class of $\langle (H-E)^{k}, j_{*}(h_{1})^{k-1} \rangle$ is $(H-E)^{r-k}$. The dual class of $\langle (H-E)^{k}, j_{*}(h_{2}^{k-2}h_{1}) \rangle$ is $H^{r-k}-j_{*}(h_{2}^{r-k-2}h_{1})$. The first two are nef by Kleiman's Transversality Theorem. The third is the class of the proper transform of a $\mathbb{P}^{k}$ that meets $L$ at a point. It is base point free in the sense of \cite{FL17b}, hence nef. Since $(H-E)^{k}, j_{*}(h_{2}^{k-1}), j_{*}(h_{2}^{k-2}h_{1})$ are all effective cycles, they are the extremal rays of $\mathrm{Eff}^{k}(W_{r})$.

\end{proof}

\subsection{Blow up $\mathbb{P}^{3}$ along rational curves on smooth quadric surfaces}

Let $\pi: Y_{d}=\mathrm{BL}_{C_{d}}\mathbb{P}^{3} \rightarrow \mathbb{P}^{3}$ where $C_{d}$ is a smooth curve on a smooth quadric surface of class $(1,d-1)$. Then $C\cong \mathbb{P}^{1}$ and $\mathrm{deg}(C)=d$. By \cite{EV81}, $N_{C_{d}/\mathbb{P}^{3}}\cong \mathcal{O}_{\mathbb{P}^{1}}(2d-1)^{\oplus 2}$ when $d\geq 3$. The exceptional divisor $j:E \hookrightarrow Y_{d}$ can be described as $E=\mathbb{P}(E_{C_{d}/ \mathbb{P}^{3}})\cong \mathbb{P}^{1}\times \mathbb{P}^{1}$. We have $\mathrm{CH}^{*}(E)=\mathbb{Z}[h_{1}, h_{2}]/(h_{1}^{2}, h_{2}^{2})$, where $h_{1}$ is the fiber of $u$ and $h_{2}$ is the fiber of the other contraction. By Proposition \ref{chowblowup}, the Chow ring of $Y_{d}$ is as follows.

\begin{proposition}\label{chowY}
When $d\geq 3$, we have $\xi=c_{1}(N_{E/ Y_{d}})= -h_{2}+(2d-1)h_{1}\in \mathrm{Pic}(E)$. The Chow ring of $Y_{d}$ is
  \[
    \mathrm{CH}^{*}(Y_{d})=
    \dfrac{
      \left[\dfrac{\mathbb{Z}[H]}{(H^{4})}\right]\oplus j_{*}\left[\dfrac{\mathbb{Z}[h_{1}, h_{2}]}{(h_{1}^{2}, h_{2}^{2})}\right]
    }
    {\left(
      \begin{array}{c}
        H\cdot j_{*}(\alpha)-j_{*}(\alpha\cdot dh_{1}),\\
        j_{*}(\beta)\cdot j_{*}(\gamma)+j_{*}(\xi\cdot \beta \cdot \gamma)
     \end{array}
    \right) },
  \quad \alpha, \beta, \gamma \in \dfrac{\mathbb{Z}[h_{1}, h_{2}]}{(h_{1}^{2}, h_{2}^{2})}.
  \]
\end{proposition}

For $d=1$, $C_{d}$ is a line, $Y_{1}$ is studied in Proposition \ref{effW} The next case is $d=2$.

\begin{proposition}\label{effconic}
  Let $h_{1}\in \mathrm{Pic}(E)$ be the fiber class of $\pi$. We have
  \begin{eqnarray*}
    \mathrm{Eff}^{1}(Y_{2})&=&\langle E, H-E \rangle, \\
    \mathrm{Eff}_{1}(Y_{2})&=& \langle j_{*}(h_{1}), H^{2}-2j_{*}(h_{1}) \rangle.
    \end{eqnarray*}
\end{proposition}

\begin{proof}
  Since $E$ is contracted by $\pi$, it spans an extremal ray of $\mathrm{Eff}^{1}(Y_{2})$. To show $H-E$ spans an extremal ray of $\mathrm{Eff}^{1}(Y_{2})$, we need to show its dual cycle $H^{2}-j_{*}(h_{1})$ is nef. Note that the proper transform of a line $L$ that meets $C_{2}$ in a point has class $H^{2}-j_{*}(h_{1})$. By moving $L$, we see that such proper transforms cover $Y_{2}$. Hence $H^{2}-j_{*}(h_{1})$ is nef.

An extremal ray of $\mathrm{Eff}_{1}(Y_{2})$ is spanned by $j_{*}(h_{1})$ since it is contracted by $\pi$. Note that $2H-E$ is nef since the ideal of $C_{2}$ is generated by quadrics, and its dual class $H^{2}-2j_{*}(h_{1})$ is the class of a 2-secant line of $C$. Hence $H^{2}-2j_{*}(h_{1})$ spans another extremal ray of $\mathrm{Eff}_{1}(Y_{2})$. 
\end{proof}

In the following we assume $d\geq 3$. The cone of effective divisors is computed in the following proposition.
   
\begin{proposition}\label{effY}
  For $d\geq 3$, we have
       $$\mathrm{Eff}^{1}(Y_{d})=\langle E, 2H-E \rangle .$$
   \end{proposition}
   
   \begin{proof}
   Since $E$ is contracted by $\pi$, it spans an extremal ray. The proper transform of any quadric containing $C_{d}$ has class $2H-E$, hence $2H-E$ is effective. The dual class of $2H-E$ is $H^{2}-2j_{*}(h_{1})$, which is the class of a 2-secant line of $C_{d}$. Since $C_{d}$ is non-degenerate for $d\geq 3$, $\mathrm{Sec}_{2}(C)=\mathbb{P}^{3}$. Hence $H^{2}-2j_{*}(h_{1})$ is nef.
   \end{proof}
   
  Next we compute the cone of effective 1-dimensional cycles. We first make an observation. 
  \begin{lemma}\label{min}
    Let $S\subset \mathbb{P}^{3}$ be an irreducible surface containing $C_{d}$. If $Q\neq S$, then $\mathrm{deg}(S)\geq d-1$. 
  \end{lemma}

  \begin{proof}
Let $S$ be such a surface with $\mathrm{deg}(S)=n$. Note that $C_{d}\subset S\cap Q= (n,n)$ in $Q$, we have $n\geq d-1$, 
  \end{proof}

  \begin{proposition}\label{eff1Y}
    For $d\geq 3$, we have
      $$\mathrm{Eff}_{1}(Y_{d})=\langle j_{*}(h_{1}), H^{2}-(d-1)j_{*}(h_{1})\rangle. $$
    \end{proposition}

    \begin{proof}
      Since $j_{*}(h_{1})$ is contracted by $\pi$, it spans an extremal ray. We need to show that the dual class of $H^{2}-(d-1)j_{*}(h_{1})$, which is $(d-1)H-E$, is nef. Consider the following exact squence on $\mathbb{P}^{3}$:
      $$0 \longrightarrow I_{C_{d}}((d-1)H) \longrightarrow \mathcal{O}_{\mathbb{P}^{3}}((d-1)H) \longrightarrow \mathcal{O}_{C_{d}}((d-1)H) \longrightarrow 0. $$
      Since $\mathrm{h}^{0}(\mathcal{O}_{\mathbb{P}^{3}}((d-1)H))=\binom{d+2}{3}$ and $\mathrm{h}^{0}(\mathcal{O}_{C}((d-1)H))=d^{2}-d+1$, we have
      $$\mathrm{h}^{0}(\mathbb{P}^{3}, I_{C_{d}}((d-1)H))-\mathrm{h}^{0}(\mathcal{O}_{\mathbb{P}^{3}}((d-1)H-Q))\geq d-1. $$
 By lemma \ref{min}, there exists an irreducible surface $S\in  \mathrm{H}^{0}(I_{C_{d}}((d-1)H))$. 

 Since $Q\cap S=(d-1, d-1)\subset Q$, we may write $Q\cap S=C_{d}\cup R$, where $R\in  \mathrm{H}^{0}(Q, \mathcal{O}_{Q}(0, d-2))$. Note that the map $ \mathrm{H}^{0}(\mathbb{P}^{3},\mathcal{O}_{\mathbb{P}^{3}}((d-1)H)) \twoheadrightarrow  \mathrm{H}^{0}(Q, \mathcal{O}_{Q}(d-1, d-1))$ is surjective implies that the map
 $$ \mathrm{H}^{0}(\mathbb{P}^{3}, I_{C_{d}}((d-1)H)) \longrightarrow  \mathrm{H}^{0}(Q, I_{C_{d}}(d-1, d-1))= \mathrm{H}^{0}(Q, \mathcal{O}_{Q}(0, d-2)) $$
 is also surjective. Hence a general member of $ \mathrm{H}^{0}(Q, \mathcal{O}_{Q}(0, d-2))$ can be realized as such an $R$. In particular, the base locus of $(d-1)H-E$ is contained in $E$.

 Let $p\in C$ be any point, we would like to show $(d-1)H-E$ has no base locus in $E$ over $p$. To do this, take $S'=Q\cup Z$ where $Z$ is any surface of degree $d-3$ that does not contain $p$. It suffices to show over $p$, the proper transforms of $S$ and $S'$ are disjoint, which is equivalent to $T_{p}Q\neq T_{p}S$. Since $R$ is a general member of $ \mathrm{H}^{0}(Q, \mathcal{O}_{Q}(0, d-2))$, we may assume it consists of $d-2$ disjoint lines $L_{1}, \cdots, L_{d-2}$ and $p\not\in R$. Let $L\in  \mathrm{H}^{0}(Q, \mathcal{O}_{Q}(0, 1))$ be a line that contains $p$. Since $L\subset Q$, $T_{p}L\subset T_{p}Q$. However, note that $L\cap S$ contains $(d-1)$ distincts points $(L\cap L_{1}), \cdots, (L\cap L_{d-2}), p$ and $\mathrm{deg}(S)=d-1$, we have $T_{p}L\not\subset T_{p}S$. Hence $T_{p}Q\neq T_{p}S$. 
\end{proof}

\subsection{Blow up $ \mathbb{P}^{3}$ along general rational curves of small degrees}

Let $\pi: Q_{d}= BL_{C_{d}}(\mathbb{P}^{3}) \rightarrow \mathbb{P}^{3}$ where $C_{d}$ is a general rational curve of degree $d$. In this subsection, we compute the cones of effective divisors in $Q_{d}$ for $1\leq d \leq 6$. Let $j_{*}: E \hookrightarrow Q_{d}$ be the exceptional divisor and $h_{1}\in \mathrm{Pic}(E)$ be the fiber class of $\pi$. 

\begin{proposition}\label{lowdegree}
We have
\begin{eqnarray*}
 \mathrm{Eff}^{1}(Q_{1})& = &\langle E, H-E \rangle,\\
\mathrm{Eff}^{1}(Q_{2})& = & \langle E, H-E \rangle,\\
\mathrm{Eff}^{1}(Q_{3})& = & \langle E, 2H-E \rangle,\\
\mathrm{Eff}^{1}(Q_{4})& = &\langle E, 2H-E \rangle,\\
\mathrm{Eff}^{1}(Q_{5})& = & \langle E, 8H-3E \rangle,\\
\mathrm{Eff}^{1}(Q_{6})& = & \langle E, 3H-E \rangle.
\end{eqnarray*}
\end{proposition}

\begin{proof}
  In all cases, $E$ is contracted by the blow up map, hence it generates an extremal ray.
  The claim for $Q_{1}$ follows from Proposition \ref{effW}. The claim for $Q_{2}$ follows from Proposition \ref{effconic}. The claim for $Q_{3}$ follows from Proposition \ref{effY}.

  For $Q_{4}$, consider the following exact sequence
  $$0 \longrightarrow I_{C_{4}}(2H) \longrightarrow \mathcal{O}_{\mathbb{P}^{3}}(2H) \longrightarrow \mathcal{O}_{C_{4}}(2H) \longrightarrow 0. $$
  Since $\mathrm{h}^{0}(\mathcal{O}_{\mathbb{P}^{3}}(2H))=10$ and $\mathrm{h}^{0}(\mathcal{O}_{C_{4}}(2H))=\mathrm{h}^{0}(\mathcal{O}_{\mathbb{P}^{1}}(8))=9$, there is a quadric surface $Q$ containing $C_{4}$. Let $X\subset Q_{4}$ be an irreducible surface whose class is $aH-bE$. Since $C_{4}$ is non-degenerate, there exists a secant line $L$ of $C_{4}$ that is not contained in $X$. Hence $[X]\cdot [L]\geq 0$, we get $a\geq 2b$. Since the proper transform of $Q$ has class $2H-E$, it generates an extremal ray.

  We first deal with $Q_{6}$. Consider the following exact sequecne
  $$0 \longrightarrow I_{C_{6}}(3H) \longrightarrow \mathcal{O}_{\mathbb{P}^{3}}(3H) \longrightarrow \mathcal{O}_{C_{6}}(3H) \longrightarrow 0. $$
  Since $\mathrm{h}^{0}(\mathcal{O}_{\mathbb{P}^{3}}(3H))=20$ and $\mathrm{h}^{0}(\mathcal{O}_{C_{6}}(3H))=19$, $C_{6}$ is contained in a cubic surface $S$.
  
  Let $T=T(C_{6})$ be the trisecant surface of $C_{6}$. Let $X\subset Q_{6}$ be an irreducible surface of class is $aH-bE$. If $X\neq T$, then there is a tri-secant line $L$ of $C_{6}$ that is not contained in $X$. We have
  $$[X]\cdot [L]=(aH-bE)\cdot (H^{2}-3j_{*}h_{1})=a-3b\geq 0. $$
  Otherwise, $X=T$. Hence the cone of effective divisors is generated by $E, 3H-E, [T]$. A classical formula of Berzolari \cite{LB82, GP82} computes the degree of the trisecant surface of a smooth space curve $C_{d, g}$ of degree $d$ and genus $g$: 
  $$\mathrm{deg}(T(C_{d,g}))=\dfrac{(d-1)(d-2)(d-3)}{3}-(d-2)g. $$
  In our case, $\mathrm{deg}(C_{6})=20$. To compute the generic multiplicity of $T$ along $C_{6}$, let $l$ be a general line that meets $C_{6}$ at a general point $p\in C_{6}$. Let $\overline{C_{6}}\subset \mathbb{P}^{2}$ be the image of $C_{6}$ under a general projection from $p$. Trisecant lines of $C_{6}$ that meet $p$ are in one to one correspondence with nodes of $\overline{C_{6}}$. Since $\overline{C_{6}}$ is a rational curve of degree 5, it has 6 nodes. We need to show that the 6 trisecant lines that meet $l$ are counted with multiplicity 1.
  
  Let $C_{6}'\subset \mathbb{P}^{6}$ be a rational normal curve and $\Lambda \subset \mathbb{P}^{6}$ be a general $\mathbb{P}^{2}$. Let $p' \in C_{6}'$ be a general point, $\widetilde{l}\subset \mathbb{P}^{6}$ be a general $\mathbb{P}^{4}$ that contains $\Lambda, p' $. Then under the projection from $\Lambda$, $C_{6}$ is the image of $C_{6}'$, $p$ is the image of $p'$, and $\widetilde{l}$ is the pre-image of a general line $l\subset \mathbb{P}^{3}$ that meets $p$. Trisecant lines $L\subset \mathbb{P}^{3}$ of $C_{6}$ are in one to one correspondence to trisecant planes $L' \subset \mathbb{P}^{6}$ that meet $\Lambda$, and $L\cap l\neq \emptyset$ if and only if $\mathrm{dim}(L'\cap \widetilde{l})\geq 1$. Let
    \begin{eqnarray*}
    \Psi &=& \{\text{Trisencant planes of }C_{6}'\}\subset G(3,7), \\
    \sigma_{2,1}(\Lambda, \widetilde{l})& = &\{L'\in G(3,7): L'\cap \Lambda \neq \emptyset, \mathrm{dim}(L'\cap \widetilde{l})\geq 1\}.
    \end{eqnarray*}
    Then $\Psi \cap \sigma_{2,1}(\Lambda, \widetilde{l})$ is in one to one correspondence with trisecant planes of $C_{6}'$ that meets $\Lambda$. By a similar tangent space computation as in the proof of Lemma \ref{nefn-1}, we see that $\Psi\cap \sigma_{2,1}(\Lambda, \widetilde{l})$ is reduced. Hence the generic multiplicity of $T$ along $C_{6}$ is 6, its class is $20H-6H$. Since $\frac{6}{20}<\frac{1}{3}$, we have $\mathrm{Eff}^{1}(Q_{6})=\langle E, 3H-E \rangle$.

    By a similar argument, the cone of effective divisors of $Q_{5}$ is generated by $E, 3H-E, [T(C_{5})]$. In this case $[T(C_{5})]=8H-3E$. Since $\frac{3}{8}>\frac{1}{3} $, we have $\mathrm{Eff}^{1}(Q_{5})=\langle E, 8H-3E \rangle$. 
  
\end{proof}

\section{Blow up $\mathbb{P}^{3}$ along curves of high degree in a fixed surface}

Let $S_{d}$ be a smooth surface of degree $d$ and $C_{e}\subset S_{d}$ be a smooth curve of degree $e$, which is not a rational curve in general.  Let $\pi: Z(C_{e})=\mathrm{BL}_{C_{e}}\mathbb{P}^{3} \rightarrow \mathbb{P}^{3}$ be the blow up and $j: E \hookrightarrow Z(C_{d})$ be the exceptional divisor. In this section we compute the cone of effective cycles in $Z(C_{d})$ when $e$ is large compared to $d$.

\begin{proposition}\label{effZ}
  When $e\geq d^{2}$, we have
  $$\mathrm{Eff}^{1}(Z(C_{d}))=\langle dH-E, E \rangle. $$
\end{proposition}

\begin{proof}
  Since $E$ is contracted by $\pi$, it spans an extremal ray. Hence it suffices to consider proper transforms $\widetilde{X}$ of irreducible surfaces $X\subset \mathbb{P}^{3}$ that is distinct from $S$. Write $[\widetilde{X}]=aH-bE$. We would like to show $a\geq db$.

  Take an integer $m\gg 0$. Since $X\neq S$, there is a smooth irreducible curve $D\in |mH-C|_{S}$ that intersects $\widetilde{X}\cap S$ transversely. Let $h_{1}\in \mathrm{Pic}(E)$ be the fiber class of $\pi$. We have
  \begin{equation}\label{Z1}
    0\leq D\cdot \widetilde{X} = [(md-e)H^{2}-j_{*}(h_{1})(D\cdot C)_{S}]\cdot (aH-bE)=(md-e)a-(D\cdot C)b.
    \end{equation}
  Note that $K_{S}=(K_{\mathbb{P}^{3}}+[S])|_{S}=(d-4)H$, we have
  \begin{equation}\label{Z2}
    (D\cdot C)_{S}=(mH-C)\cdot C=me-[(2p_{g}(C)-2)-K_{S}\cdot C]=(m+d-4)e-(2p_{g}(C)-2).
    \end{equation}
    Let $f: \mathbb{P}^{3}\dashrightarrow \mathbb{P}^{2}$ be the projection from a general point. Then
    $$2p_{g}(C)-2=2p_{g}(f(C))-2\leq 2p_{a}(f(C))-2=e(e-3). $$
    By (\ref{Z1}) and (\ref{Z2}), we have
    \[
      a  \geq  \left(\dfrac{me-K_{S}\cdot C -(2p_{g}-2)}{md-e}\right) \cdot b
       \geq \left(\dfrac{me -(d-4)e -e(e\cdot 3)}{md}\right) \cdot b.
    \]
    Taking limit as $m \rightarrow \infty$, we see that $a\geq (e/d)\cdot b$. When $e\geq d^{2}$, we have $a\geq db$. 
\end{proof}

\bibliographystyle{alpha}
\bibliography{cones.bib}

\end{document}